\documentclass[preprint,11pt]{elsarticle}
\usepackage{amsfonts}
\usepackage{amsmath,amssymb,amsthm}
\usepackage{graphicx}
\usepackage{color}
\usepackage{soul} 
\usepackage{color, xcolor} 

\makeatletter
\def\ps@pprintTitle{%
  \let\@oddhead\@empty
  \let\@evenhead\@empty
  \def\@oddfoot{\reset@font\hfil\thepage\hfil}
  \let\@evenfoot\@oddfoot}
\makeatother

\newtheorem{defin}{Definition}[section]
\newenvironment{definition}{\begin{defin}\rm}{\end{defin}}
\newtheorem{theorem}[defin]{Theorem}
\newtheorem{lemma}[defin]{Lemma}
\newtheorem{proposition}[defin]{Proposition}
\newtheorem{corollary}[defin]{Corollary}
\newtheorem{qu}{Question}

\begin{document}

\begin{frontmatter}

\title{The Hartman--Mycielski construction in topological MV-algebras\corref{cor2}}


\author[mymainaddress]{Li-Hong Xie}
\ead{yunli198282@126.com}

\author[mymainaddress1]{Jiang Yang\corref{cor1}} 
\ead{yangjiangdy@126.com}

\cortext[cor1]{Corresponding author: Jiang Yang}
\cortext[cor2]{The project is supported by the Natural Science Foundation of Guangdong Province under Grant (No. 2021A1515010381) and the Innovation Project of Department of Education of Guangdong Province (No. 2022KTSCX145).}
\address[mymainaddress]{School of Mathematics and Computational Science, Wuyi University, Jiangmen, Guangdong, 529000, P.R. China}

\address[mymainaddress1]{ School of Mathematical Sciences, Guangxi Minzu University, Nanning, 530006, P.R. China}

\begin{abstract}
Recently, topological MV-algebras have been investigated by several mathematicians. In this paper, we mainly show that for every Hausdorff topological MV-algebra $A$,
 there exists a natural topological
isomorphism $i_A:A\rightarrow A^\bullet$ of $A$ onto a closed subalgebra of the pathwise connected, locally
pathwise connected topological MV-algebra $A^\bullet$. Furthermore, we show that there is an extension to a bounded continuous function
on the MV-algebra $A^\bullet$ for each continuous real-valued bounded
 function on a topological MV-algebra $A$. Finally, we prove that if $\varphi:A_1\rightarrow A_2$ is a continuous homomorphism of
 topological MV-algebras, then $\varphi$ admits a natural extension to a
continuous homomorphism $\varphi^\bullet:A_1^\bullet\rightarrow A_2^\bullet$; in addition, if $\varphi$ is open and onto, then so is $\varphi^\bullet$.
\end{abstract}


\end{frontmatter}

\section{Introduction}
In recent decades, a number of algebraic structures associated with logical systems have been studied by many mathematicians.
 For instance, Y. Imai and K. Is\'{e}ki \cite{Imai} introduced BCK-algebras as an algebraic formulation of Meredith's BCK-implicational calculus.
  To investigate many-valued logic by algebraic means, BL-algebras have been defined by
H\v{a}jek \cite{Haj}. Undoubtedly, MV-algebras, which were introduced by Chang \cite{Chang} in order to show {\L}ukasiewicz logic to be standard complete,
are among the most important structures associated with logical systems, where ``MV'' is short for ``many-value''. Furthermore, a number of algebraic structures
associated with logical systems endowed with a topology have been investigated by several mathematicians \cite{Bor,Hav, Rou,Yang}.
In particular, Hoo \cite{Hoo} introduced the notion of a topological MV-algebra, which means an MV-algebra $(A,\oplus,\ast,0)$ with a topology such that the operations
$\oplus$ and $\ast$ are continuous functions. Some fundamental properties were investigated by Hoo. In 2012, Weber \cite{Web} obtained a representation theorem for complete MV-algebras
endowed with a Hausdorff order continuous locally convex topology which admits a 0-neighbourhood base
consisting of sublattices. In 2017, Najafi, Rezaei and Kouhestani \cite{Naj} introduced some weaker
 versions of topological MV-algebras such as (para, quasi, semi)topological MV-algebras. Later, MV-algebras endowed with filter topologies were studied by
  Luan-Yang \cite{Luan-Yang}. Asadzadeh-Rezaei-Jamalzadeh \cite{ARJ}, Wu-Luan-Yang
\cite{WLY}, Luan-Zhao-Yang \cite{LZY}, and Luan-Weber-Yang \cite{LWY} further developed filter topology theory on MV-algebras. Recently, Gan-Luan-Deng-Yang \cite{GLD}
proved that the separation axioms $T_0, T_1, T_2, T_3$ are equivalent for topological MV-algebras. However, whether every $T_0$ topological MV-algebra is Tychonoff is still open. In 2026,
Li-Yang \cite{LY} gave a positive answer to this question in the locally convex topological MV-algebras. It is well known that every Hausdorff topological group $G$
is topologically isomorphic to a closed subgroup of a connected, locally connected topological group \cite{HM}. Motivated by these results, it is natural to investigate
whether every topological MV-algebra also possesses those properties, specifically, whether every Hausdorff topological MV-algebra can be topologically isomorphic to a closed
subalgebra of a connected, locally connected topological MV-algebra.

The paper is organized as follows. After some preliminaries on MV-algebras in Section 2,  we show in Section 3 that for every Hausdorff topological MV-algebra $A$,
 there exists a natural topological
isomorphism $i_A:A\rightarrow A^\bullet$ of $A$ onto a closed subalgebra of the pathwise connected, locally
pathwise connected topological MV-algebra $A^\bullet$ (see Theorem \ref{the1}). Furthermore, we show that there is an extension to a bounded continuous function
on the MV-algebra $A^\bullet$ for each continuous real-valued bounded
 function on a topological MV-algebra $A$ (see Corollary \ref{cor1}). Finally, we prove that if $\varphi:A_1\rightarrow A_2$ is a continuous homomorphism of
 topological MV-algebras, then $\varphi$ admits a natural extension to a
continuous homomorphism $\varphi^\bullet:A_1^\bullet\rightarrow A_2^\bullet$; in addition, if $\varphi$ is open and onto, then so is $\varphi^\bullet$
(see Theorem \ref{The3}).

\section{Preliminaries}\label{Sec:2}

In this section, we collect some relevant definitions and results from $MV$-algebras to make this paper self-contained.

\begin{definition}\label{Def:1}\cite{Chang,CD,Mun}
An $MV$-algebra is an algebra $(A,\oplus,\ast, 0)$ of type $(2, 1, 0)$ such that for all $x,y\in A$,
\begin{enumerate}
\item[(MV1)] $(A,\oplus, 0)$ is a commutative monoid;
\item[(MV2)] $x^{\ast\ast}=x$;
\item[(MV3)] $x\oplus 0^\ast=0^\ast$;
\item[(MV4)] $(x^\ast\oplus y)^\ast\oplus y=(y^\ast\oplus x)^\ast\oplus x$.
\end{enumerate}
\end{definition}

In an $MV$-algebra $A$ for any $x, y \in A$ we define:
\begin{enumerate}
\item[(MV5)] $1 := 0^*$;
\item[(MV6)] $x \odot y := (x^* \oplus y^*)^*$;
\item[(MV7)] $x \ominus y := x \odot y^*$.
\end{enumerate}

In an $MV$-algebra \( A \), for any \( x, y \in A \) we define \( x \leq y \) if and only if \( x^* \oplus y = 1 \). It is well known \cite{CD} that \(\leq\) is
a partial order relation on \( A \), which determines a distributive lattice structure, where the join
\[ x \vee y := y \oplus (x \ominus y) ,\]
the meet
\[ x \wedge y := x \odot(x^* \oplus y), \] where 0 is the smallest element and 1 is the greatest element. By (MV6) and (MV7), for
any \( x, y \in A \),
\[ x \leq y \iff x \ominus y = 0. \]

\begin{lemma}\cite{Chang, CD, Mun}\label{Lema1}
 Let $A$ be an MV-algebra. For any $x,y,z\in A$, the following statements hold:
 \begin{enumerate}
 \item[(1)] $ x =(x \wedge y) \oplus(x\ominus y);$
 \item[(2)] $z\odot(x\vee y)=(z\odot x)\vee (z\odot y);$
 \item[(3)] $x\odot x^\ast=0;$
 \item [(4)] $x\leq y\Leftrightarrow y=x\vee y$
 \end{enumerate}
\end{lemma}

\begin{definition}\cite[Definition 4.1]{Chang}\label{Def2}
 A subset $I$ of an MV-algebra $A$ is called an ideal of $A$ if
 \begin{enumerate}
 \item[(1)] $0\in I$;
 \item[(2)] $x\oplus y\in I$ for each $x,y\in I$;
 \item[(3)] $c\in I$ implies that $z\in I$ for any $z\in A$ with $z\leq c$.
 \end{enumerate}
\end{definition}

Let $A$ be an MV-algebra. Recall that the {\bf distance function} $d:A\times A\rightarrow A$ is defined by $d(x,y)=(x\ominus y)\oplus(y\ominus x)$,
where $x\ominus y=x\odot y^\ast$(\cite[Definition 1.2.4]{CD}
or \cite[ Definition 2.0.11]{Mun}).

\begin{lemma}(\cite[Proposition 1.2.6]{CD}
or \cite[Proposition 2.0.13]{Mun})
 Let $I$ be an ideal of an MV-algebra $A$. Define an binary relation $\equiv_I$ on $A$ by: $x\equiv_I y\Leftrightarrow d(x,y)\in I$. Then $\equiv_I$ is
  a congruence. Moreover, $I=\{x\in A: x\equiv_I 0\}$.
\end{lemma}

Let $I$ be an ideal of an MV-algebra $A$. Given $x\in I$, the equivalent class of $x$ with respect to $\equiv_I$ will be denoted by $x/I$ and the
quotient set $A/\equiv_I$ by $A/I$. Since $\equiv_I$ is a congruence, defining on the set $A/I$ the operations:

\[(x/I)^\ast =x^\ast/I \text{~and~} x/I\oplus y/I=(x\oplus y)/I\]
the structure $(A/I, \oplus,\ast, 0/I)$ becomes an MV-algebra, called the {\bf quotient algebra} of $A$ by the ideal $I$. Moreover, the correspondence
$x\mapsto x/I$ defines a homomorphism $\rho_I$ from $A$ onto the quotient algebra $A/I$, which is called the {\bf natural homomorphism} from $A$
onto $A/I$ \cite{CD,Mun}.

Let \( A \) and \( B \) be $MV$-algebras. Recall that a mapping \( f : A \to B \) is a \textbf{homomorphism} \cite{CD} if it satisfies the following conditions: for all \( x, y \in A \),

\begin{enumerate}
    \item[(i)] \( f(0) = 0 \);
    \item[(ii)] \( f(x \oplus y) = f(x) \oplus f(y) \);
    \item[(iii)] \( f(x^*) = (f(x))^* \).
\end{enumerate}

If \( f \) is one-to-one, we say that \( f \) is \textbf{injective}, or an \textbf{embedding}. If \( f \) is onto, we say that \( f \) is \textbf{surjective}. By an \textbf{isomorphism}, we mean a surjective and one-to-one homomorphism. We write \( A \cong B \) if there exists an isomorphism from \( A \) onto \( B \).

\begin{definition}\cite{Hoo}
Let $A$ be an $MV$-algebra with a topology $\tau$. Then $(A,\tau )$ is called a topological MV-algebra if the operations $\oplus$ and $\ast$ are continuous.
\end{definition}
 Let $A$ be an $MV$-algebra. Given \( a \in A \) and \( U \subseteq A \), denote by

\[
U(a) = \{ x \in A \mid a \ominus x \in U \text{ and } x \ominus a \in U \};
\]
\[
U^{(a)} = \{ x \in A \mid a \lor x \in U \text{ and } a \ominus x \in U \}.
\]

It is clear that \( U(a) \subseteq V(a) \) and \( U^{(a)} \subseteq V^{(a)} \) whenever \( U \subseteq V \subseteq A \).

The following result shows that the topology \( \tau \) of a topological $MV$-algebra \( (A, \tau) \) is uniquely determined by its \( 0 \)-neighborhood system.

\begin{proposition}\label{Pro1}\cite[Proposition 3.5]{GLD}
\label{prop:3.5}
Let \((A, \tau)\) be a topological $MV$-algebra and \(\mathcal{V}\) an open neighborhood base of \(0\). Then \(\mathcal{V}\) satisfies the following conditions:

\begin{enumerate}
    \item[(i)] \(0 \in U\) for every \(U \in \mathcal{V}\);
    \item[(ii)] for every \(U, V \in \mathcal{V}\) and \(a, b \in A\), if \(U(a) \cap V(b) \neq \emptyset\), then for each \(y \in U(a) \cap V(b)\), there exists \(W \in \mathcal{V}\) such that \(W(y) \subseteq U(a) \cap V(b)\);
    \item[(iii)] for every \(U \in \mathcal{V}\) and \(a, b \in A\), there exists \(V \in \mathcal{V}\) such that \(V(a) \oplus V(b) \subseteq U(a \oplus b)\).
\end{enumerate}

Conversely, let \(A\) be an MV-algebra, and let \(\mathcal{V}\) be a family of subsets of \(A\) satisfying conditions (i), (ii) and (iii). Then the family
\[
B_{\mathcal{V}} = \{ U(x) \mid x \in A,\; U \in \mathcal{V} \}
\]
is a base for a topology on \(A\). With this topology, \(A\) is a topological $MV$-algebra.
\end{proposition}

\section {Main results}

{\bf Construction of Hartman--Mycielski:} Let $(A,\oplus,\ast, 0)$ be an MV-algebra and let $J = [0, 1)$. A function $f : J\rightarrow A$ is a step function
if there are real numbers $b_0, b_1,\ldots, b_n$ such that $0 = b_0 < b_1 < \cdots< b_n = 1$ and $f$ is constant on
$[b_k, b_{k+1})$ for all $k = 0,1,\ldots,n-1$.
 Henceforward, when we say that $B=\{b_0, b_1,\ldots, b_n\}$ is a partition
of $J$, we include the condition that $0 = b_0 < b_1 < \cdots< b_n = 1$. Denote by $A^\bullet$ the set of all step
functions. Define two operations $\oplus$ and $\ast$ on $A^\bullet$ by
$$( f \oplus g)(r) = f (r) \oplus g(r), \quad r\in J\quad\quad\quad(1)$$
$$f^\ast(r) = (f (r))^\ast, \quad r\in J\quad\quad\quad \quad\quad \quad\quad(2)$$
for all $f,g \in A^\bullet$. Let $f, g \in A^\bullet$. It is easy to see that both $f \oplus g$ and $f^\ast$ are again step functions.

\begin{proposition}\label{Pro2}
For every MV-algebra $(A,\oplus,\ast, 0)$, $(A^\bullet,\oplus,\ast, 0^\bullet)$ forms an MV-algebra in Construction of Hartman--Mycielski, where $0^\bullet(r)=0$ for each $r\in J=[0,1)$.
\end{proposition}

\begin{proof}
Let us show that $A^\bullet$ satisfies the conditions (MV1)-(MV4) in Definition \ref{Def:1}. Since $(A,\oplus, 0)$ is a commutative monoid, one can easily show that so
is $(A^\bullet,\oplus, 0^\bullet)$.

For (MV2), take any $f\in A^\bullet$. According to the definition of the operation $\ast$ on $A^\bullet$, for each $r\in J$ we have that
 \[f^{\ast\ast}(r)=(f^\ast(r))^\ast=((f(r))^\ast)^\ast=f(r)^{\ast\ast}=f(r),\]
 since $A$ satisfies the condition (MV2) in Definition \ref{Def:1}. This implies that $f^{\ast\ast}=f$.

For (MV3), take any $f\in A^\bullet$. Then, for each $r\in J$,
\begin{align*}
(f\oplus (0^\bullet)^\ast)(r)&=f(r)\oplus (0^\bullet)^\ast(r)
\\&=f(r)\oplus (0^\bullet(r))^\ast
\\&=f(r)\oplus 0^\ast
\\&=0^\ast
\\&=(0^\bullet)^\ast(r),
\end{align*}
which implies that $f\oplus (0^\bullet)^\ast=(0^\bullet)^\ast$.

For (MV4), take any $f,g\in A^\bullet$. Then, for each $r\in J$,
\begin{align*}
((f^\ast\oplus g)^\ast\oplus g)(r)&=(f^\ast\oplus g)^\ast(r)\oplus g(r)
\\&=(f(r)^\ast\oplus g(r))^\ast\oplus g(r)
\\&=(g(r)^\ast\oplus f(r))^\ast\oplus f(r)
\\&=(g^\ast\oplus f)^\ast(r)\oplus f(r)
\\&=((g^\ast\oplus f)^\ast\oplus f)(r),
\end{align*}
which implies that $(f^\ast\oplus g)^\ast\oplus g=(g^\ast\oplus f)^\ast\oplus f$.
\end{proof}

Let $A$ be a topological $MV$-algebra. Using the sufficient conditions of Proposition \ref{Pro1}, we can topologize the $A^\bullet$ which is a connected, locally connected
topological MV-algebra. Given an open neighborhood $U$ of $0$ in $A$ and a real number $\varepsilon> 0$, define

$$O(U,\varepsilon)=\{f\in A^\bullet\mid \mu\{r\in J\mid f(r)\notin U\} <\varepsilon\},$$
where $\mu$ is the Lebesgue measure on the real line.

\begin{proposition}\label{Pro3}
Let $A$ be a topological $MV$-algebra and $\mathcal{U}$ an open neighborhood base of $0$. Put $\mathcal{V}=\{O(U,\varepsilon)\mid U\in \mathcal{U},\varepsilon>0\}$. Then, the family
$$\mathcal{B}=\{O(U,\varepsilon)(f)\mid O(U,\varepsilon)\in\mathcal{V}, f\in A^\bullet\}$$
forms a base of a topology on $A^\bullet$, and $A^\bullet$ with the topology generated by $\mathcal{B}$ becomes a topological $MV$-algebra.
\end{proposition}

\begin{proof}
According to Proposition \ref{Pro2} $(A^\bullet,\oplus,\ast, 0^\bullet)$ forms an $MV$-algebra. To show that $A$ with the topology generated by $\mathcal{B}$ becomes a topological $MV$-algebra, it is enough to show that $\mathcal{V}$ satisfies the conditions $(i)$-$(iii)$ in Proposition \ref{Pro1}.

Since $0^\bullet(r)=0$ holds for each $r\in J$, $0^\bullet\in O(U,\varepsilon)$ holds for each $U\in \mathcal{U}$ and $\varepsilon>0$. Thus,  $\mathcal{V}$ satisfies the condition $(i)$.

For $(ii)$. Take an  $O(U,\varepsilon)\in \mathcal{V}$ and fix $g\in A^\bullet $. For $f\in O(U,\varepsilon)(g)$, since  \[
 O(U,\varepsilon)(g)=\{l\in A^\bullet\mid g\ominus l\in O(U,\varepsilon) \text{~and~} l\ominus g\in O(U,\varepsilon)\},
  \]
  we have that
  \[
  \mu\big(\{t\in J\mid (g\ominus f)(t)\notin U\}\big)<\varepsilon,\qquad
  \mu\big(\{t\in J\mid (f\ominus g)(t)\notin U\}\big)<\varepsilon.
  \]
Let \(0=a_0<a_1<\dots<a_n=1\) be a partition of $J$ such that for each $1\leq k\leq n$ both $f$ and $g$ are constant on each interval \(J_k=[a_{k-1},a_k)\) and are equal to \(x_k,y_k\in A\) on this interval, respectively. Put
\[
E_1=\bigcup_{\substack{k\\ x_k\ominus y_k\notin U}} J_k,\qquad
E_2=\bigcup_{\substack{k\\ y_k\ominus x_k\notin U}} J_k.
\]
Then
\[
\mu(E_1)=\mu\big(\{t\in J\mid (f\ominus g)(t)\notin U\}\big)<\varepsilon,\quad
\mu(E_2)=\mu\big(\{t\in J\mid (g\ominus f)(t)\notin U\}\big)<\varepsilon.
\]

Put

\[
N_1=\{1 \leq k\leq n\mid x_k\ominus y_k\in U\},\qquad N_2=\{1 \leq k\leq n\mid y_k\ominus x_k\in U\}.
\]

 Fix a $k\in N_1$. Since $\ominus$ is continuous and $x_k\ominus y_k\in U$, there is an open neighborhood $W_k^1\in \mathcal{U}$ such that $x \ominus y_k\in U$ for each $x\in W_k^1(x_k)$. Similarly, for each $k\in N_2$ there is an open neighborhood $W_k^2\in \mathcal{U}$ such that $y_k\ominus x\in U$ for each $x\in W_k^2(x_k)$. Choose $W_0\in\mathcal{U}$ such that
\[
W_0 \subseteq \bigcap_{k\in N_1} W_k^1 \cap \bigcap_{k\in N_2} W_k^2
\]

Since $A$ is a topological $MV$-algebra, for $U$ there is an open neighborhood $W_1\in \mathcal{U}$ such that \(W_1\oplus W_1\subseteq U\). Choose $W\in\mathcal{U}$ such that
\[
W \subseteq W_0\cap W_1.
\]

Then \(W(x_k)\subseteq \{x\in A\mid x \ominus y_k\in U\}\) for each $k\in N_1$ and  \(W(x_k)\subseteq \{x\in A\mid y_k\ominus x\in U\}\) for each $k\in N_2$.

Put \[\delta = \frac{1}{2}\big(\varepsilon - \max(\mu(E_1),\mu(E_2))\big) >0.\] Since \(\mu(E_1),\mu(E_2)<\varepsilon\), we have that \(\delta>0\).

Take any \(h\in O(W,\delta)(f)\). Then
 \[
\mu\big(\{t\in J\mid (h\ominus f)(t)\notin W\}\big)<\delta,\qquad
\mu\big(\{t\in J\mid (f\ominus h)(t)\notin W\}\big)<\delta.
\]
Put
\[
F_1=\{t\in J\mid (h\ominus f)(t)\notin W\},\qquad F_2=\{t\in J\mid (f\ominus h)(t)\notin W\}.
\]
Then \(\mu(F_1)<\delta\) and \(\mu(F_2)<\delta\).

Take any \(t\notin E_1 \cup F_1 \cup F_2\). Since $t\notin E_1$, then there is a $k(t)\in N_1$ such that $t\in [a_{k(t)-1},a_{k(t)})$.
Since $t\notin F_1$, we obtain that \((h\ominus f)(t)=h(t)\ominus x_{k(t)}\in W\). Similarly, we have that \((f\ominus h)(t)=x_{k(t)}\ominus h(t)\in W\)
by $t\notin F_2$. This implies that $h(t)\in W(x_{k(t)})$. Since we have proved that \(W(x_{k}) \subseteq \{x\in A\mid x \ominus y_k\in U\}\) for each $k\in N_1$, we obtain that \(W(x_{k(t)}) \subseteq \{x\in A\mid x \ominus y_{k(t)}\in U\}\). This implies that \(h(t)\in \{x\in A\mid x \ominus y_{k(t)}\in U\}\). Thus, we obtain that $(h\ominus g)(t)=h(t)\ominus g(t)=h(t)\ominus y_{k(t)}\in U$. This implies that $t\notin \{j\in J\mid (h\ominus g)(j)\notin U\}$. Thus,
\[\{j\in J\mid (h\ominus g)(j)\notin U\}\subseteq E_1 \cup F_1 \cup F_2,\]
which implies that
\begin{align*}
\mu\big(\{j\in J\mid (h\ominus g)(j)\notin U\}\big)&\leq \mu \big(E_1 \cup F_1 \cup F_2\big)
\\&\leq\mu (E_1)+\mu (F_1)+\mu (F_2)
\\&< \mu(E_1)+2\delta\leq \mu(E_1)+\varepsilon-\mu(E_1)
\\&=\varepsilon.
\end{align*}
This implies that \[h\ominus g\in O(U,\varepsilon).\]

Take any \(t\notin E_2 \cup F_1 \cup F_2\). Since $t\notin E_2$, Then there is a $k(t)\in N_2$ such that $t\in [a_{k(t)-1},a_{k(t)})$. Similarly, we can easily show that $h(t)\in W(f(t))=W(x_{k(t)})$ by $t\notin F_1\cup F_2$. Since we have proved that \(W(x_k)\subseteq \{x\in A\mid y_k\ominus x\in U\}\) for each $k\in N_2$, we obtain that \(W(x_{k(t)}) \subseteq \{x\in A\mid y_{k(t)}\ominus x\in U\}\). This implies that \(h(t)\in \{x\in A\mid y_{k(t)}\ominus x\in U\}\). Thus, we obtain that $(g\ominus h)(t)=g(t)\ominus h(t)=y_{k(t)}\ominus h(t)\in U$. This implies that $t\notin \{j\in J\mid (g\ominus h)(j)\notin U\}$. Thus,
\[\{j\in J\mid (g\ominus h)(j)\notin U\}\subseteq E_2 \cup F_1 \cup F_2,\]
which implies that
\begin{align*}
\mu\big(\{j\in J\mid (g\ominus h)(j)\notin U\}\big)&\leq \mu \big(E_2 \cup F_1 \cup F_2\big)
\\&\leq\mu (E_2)+\mu (F_1)+\mu (F_2)
\\&< \mu(E_2)+2\delta\leq \mu(E_2)+\varepsilon-\mu(E_2)
\\&=\varepsilon.
\end{align*}
This implies that \[g\ominus h\in O(U,\varepsilon).\]

Noting that we have proved that \[h\ominus g\in O(U,\varepsilon).\] above, so we obtain that \(h\in O(U,\varepsilon)(g) \). This implies that \(O(W,\delta)(f)\subseteq O(U,\varepsilon)(g)\).

For $(iii)$, take any $f_1$, $f_2\in A^\bullet$ and $O(U,\varepsilon)\in \mathcal{V}$. Let \(0=a_0<a_1<\dots<a_n=1\) be a partition of $J$ such that for each $1\leq k\leq n$ both $f_1$ and $f_2$ are constant on each interval \(J_k=[a_{k-1},a_k)\) and are equal to \(x_k,y_k\in A\) on this interval, respectively. Put $d_k=x_k\oplus y_k$ for each $1\leq k\leq n$. For each $1\leq k\leq n$, consider the function $G_k:A\times A\rightarrow A$ defined as $G_k(x,y)=d_k \ominus(x\oplus y)$ for each $(x,y)\in A\times A$. From the fact that $A$ is a topological $MV$-algebra it follows that $G_k$ is continuous. Since $G_k(x_k,y_k)=d_k\ominus(x_k\oplus y_k)=0\in U$ and $U$ is open, there is an open neighbourhood $W_k^1\in \mathcal{U}$ such that \[G_k(x,y)=d_k\ominus(x\oplus y)\in U\]
holds for each $x\in W_k^1(x_k)$ and $y\in W_k^1(y_k)$. Similarly, consider the function $F_k:A\times A\rightarrow A$ defined as $F_k(x,y)=(x\oplus y)\ominus d_k$ for each $(x,y)\in A\times A$. We can find an open neighbourhood $W_k^2\in \mathcal{U}$ such that \[F_k(x,y)=(x\oplus y)\ominus d_k\in U\]
holds for each $x\in W_k^2(x_k)$ and $y\in W_k^2(y_k)$. Put $W_k= W_k^1\cap W_k^2 $. Then one can easily show that \[d_k\ominus(x\oplus y)\in U \text{ and }(x\oplus y)\ominus d_k\in U\]
hold for each $x\in W_k(x_k)$ and $y\in W_k(y_k).$

Put $W_1=\bigcap_{k=1}^{n}W_k$. Then one can easily show that \[d_k\ominus(x\oplus y)\in U \text{ and }(x\oplus y)\ominus d_k\in U\] hold for each $1\leq k\leq n$ and each $x\in W_1(x_k)$ and $y\in W_1(y_k),$ because $ W_1(x_k)\subseteq W_k(x_k)$ and $ W_1(y_k)\subseteq W_k(y_k)$ hold for each $1\leq k\leq n.$

Put $\delta_1=\frac{\varepsilon}{4}$. Then we have the following Claim :

{\bf Claim:} $O(W_1,\delta_1)(f_1)\oplus O(W_1,\delta_1)(f_2)\subseteq O(U,\varepsilon)((f_1\oplus f_2)).$

In fact, it is enough to show that \[(f_1\oplus f_2)\ominus (g_1\oplus g_2)\in O(U,\varepsilon) \text{~and~} (g_1\oplus g_2)\ominus (f_1\oplus f_2) \in O(U,\varepsilon)\] hold for each $g_1\in O(W_1,\delta_1)(f_1)$ and $g_2\in O(W_1,\delta_1)(f_2).$

Since $g_i\in O(W_1,\delta_1)(f_i)$ for $i=1,2$, we have that \[\mu \big( \{t\in J\mid (g_i\ominus f_i)(t)\notin W_1\}\big)< \delta_1 \text{~and~}\mu \big( \{t\in J\mid (f_i\ominus g_i)(t)\notin W_1\}\big)< \delta_1.\]
Put \[E_1=\{t\in J\mid (g_1\ominus f_1)(t)\notin W_1\}\cup \{t\in J\mid (f_1\ominus g_1)(t)\notin W_1\}\]
and \[E_2=\{t\in J\mid (g_2\ominus f_2)(t)\notin W_1\}\cup \{t\in J\mid (f_2\ominus g_2)(t)\notin W_1\}.\]
Let $E=E_1\cup E_2$. Then $\mu(E)\leq \mu(E_1)+\mu(E_2)<4\delta_1.$

Take any $t\notin E$. Then there is $1\leq k(t)\leq n$ such that $t\in [a_{k(t)-1},a_{k(t)})$. Then \[(g_i\ominus f_i)(t)\in W_1 \text{~and~} (f_i\ominus g_i)(t)\in W_1\] for $i=1,2$. This implies that
 \[g_i(t)\ominus f_i(t)\in W_1 \text{~and~} f_i(t)\ominus g_i(t)\in W_1,\]
which implies that \[g_1(t)\in W_1(f_1(t))=W_1(x_{k(t)}) \text{~and~} g_2(t)\in W_1(f_2(t))=W_1(y_{k(t)})\]
Thus, from the fact that  \[d_k\ominus(x\oplus y)\in U \text{~and~}(x\oplus y)\ominus d_k\in U\] hold for each $1\leq k\leq n$ and each $x\in W_1(x_k)$ and $y\in W_1(y_k),$  it follows that \[d_{k(t)}\ominus(g_1(t)\oplus g_2(t))\in U \text{~and~}(g_1(t)\oplus g_2(t))\ominus d_{k(t)}\in U\].

Since $t\in [a_{k(t)-1},a_{k(t)})$, we have that $f_1(t)=x_{k(t)}$ and $f_2(t)=y_{k(t)}$. Noting that $d_{k(t)}=x_{k(t)}\oplus y_{k(t)}$, so we have that
\[(f_1(t)\oplus f_2(t))\ominus(g_1(t)\oplus g_2(t))\in U \text{~and~}(g_1(t)\oplus g_2(t))\ominus (f_1(t)\oplus f_2(t))\in U.\]
That is,
\[((f_1\oplus f_2)\ominus(g_1\oplus g_2))(t)\in U \text{~and~}((g_1\oplus g_2)\ominus (f_1\oplus f_2))(t)\in U,\]
which implies that \[t\notin \{l\in J\mid ((f_1\oplus f_2)\ominus(g_1\oplus g_2))(l)\notin U \}\] and \[t\notin \{l\in J\mid ((g_1\oplus g_2)\ominus (f_1\oplus f_2))(l)\notin U\}.\]
Thus, \[\{l\in J\mid ((f_1\oplus f_2)\ominus(g_1\oplus g_2))(l)\notin U \}\subseteq E\text{~and~}\{l\in J\mid ((g_1\oplus g_2)\ominus (f_1\oplus f_2))(l)\notin U\}\subseteq E.\]
Then
\[\mu\big(\{l\in J\mid ((f_1\oplus f_2)\ominus(g_1\oplus g_2))(l)\notin U \}\big)\leq \mu(E)<4\delta_1=\varepsilon\]and\[\mu\big(\{l\in J\mid ((g_1\oplus g_2)\ominus (f_1\oplus f_2))(l)\notin U\}\big)\leq \mu(E)<4\delta_1=\varepsilon,\]
which implies that
\[(f_1\oplus f_2)\ominus(g_1\oplus g_2)\in O(U,\varepsilon)\text{~and~}(g_1\oplus g_2)\ominus (f_1\oplus f_2) \in O(U,\varepsilon).\]
\end{proof}

\begin{proposition}\label{Pro4}
The topological $MV$-algebra $A^\bullet$ constructed in Proposition \ref{Pro3} is pathwise
connected and locally pathwise connected.
\end{proposition}

\begin{proof}
Let $\mathcal{U}$ be an open neighborhood base of \(0\) in $A$. The local pathwise connectedness of $A^\bullet$ will follow if we show that each
set $O(U,\varepsilon)(g)$ is pathwise connected for each $U\in \mathcal{U} $, $g\in A^\bullet$ and $\varepsilon>0$. Take any $f\in O(U,\varepsilon)(g)$.
We claim that there is a continuous function $\varphi:[0,1]\rightarrow O(U,\varepsilon)(g)$ such that $\varphi(0)=g$ and $\varphi(1)=f$. According
to the definition of $A^\bullet$. There is a partition \(0=a_0<a_1<\dots<a_n=1\) of $J=[0,1)$ such that for each $1\leq i\leq n$ both $f$ and
$g$ are constant on each interval \(J_i=[a_{i-1},a_i)\) and are equal to \(x_i,y_i\in A\) on this interval, respectively. For each $t\in [0,1]$
and every non-negative $k<n$, set $b_{k,t}=a_k+t(a_{k+1}-a_k)$. Then $b_{k,0}=a_k$, $b_{k,1}=a_{k+1}$ for each non-negative $k<n$ and $a_{k}<b_{k,t}<a_{k+1}$
for each $t\in (0,1)$. Now we define a function $\varphi:[0,1]\rightarrow O(U,\varepsilon)(g)$ by $\varphi(0)=g$ and $\varphi(1)=f$ and, for $t\in (0,1)$
and $r\in [0,1)$,

\[\varphi(t)(r)=\left\{\begin {array}{r@{\quad \quad}l}f(r),& a_k\leq r<b_{k,t}\\
g(r),& b_{k,t}\leq r<a_{k+1}.\end
{array}\right.
\]

Firstly, we claim that $\varphi(t)\in O(U,\varepsilon)(g)$ holds for each $t\in [0,1]$. Indeed, fix non-negative $k<n$, put \[E_k=\{r\in [a_k,a_{k+1})\mid \varphi(t)(r)\ominus g(r)\notin U\}\]
and
\[F_k=\{r\in [a_k,a_{k+1})\mid g(r)\ominus\varphi(t)(r)\notin U\}.\]
Then, according to the definition of $\varphi(t)$ one can easily show that \[E_k\subseteq \{r\in [a_k,a_{k+1})\mid f(r)\ominus g(r)\notin U\}\]
and
\[F_k\subseteq \{r\in [a_k,a_{k+1})\mid g(r)\ominus f(r)\notin U\}.\]
This implies that \[\{r\in J\mid \varphi(t)(r)\ominus g(r)\notin U\}\subseteq \{r\in J\mid f(r)\ominus g(r)\notin U\}\]
and
\[\{r\in J\mid g(r)\ominus\varphi(t)(r)\notin U\}\subseteq \{r\in J\mid g(r)\ominus f(r)\notin U\},\]
which implies that\[\mu\big(\{r\in J\mid \varphi(t)(r)\ominus g(r)\notin U\}\big)\leq \mu \big(\{r\in J\mid f(r)\ominus g(r)\notin U\}\big)<\varepsilon\]
and
\[\mu\big(\{r\in J\mid g(r)\ominus\varphi(t)(r)\notin U\}\big)\leq \mu \big(\{r\in J\mid g(r)\ominus f(r)\notin U\}\big)<\varepsilon.\]
Thus, $\varphi(t)\ominus g\in O(U,\varepsilon)$ and $g\ominus\varphi(t)\in O(U,\varepsilon)$, which implies that \[\varphi(t)\in O(U,\varepsilon)(g).\]

Secondly, we claim that $\varphi$ is continuous. Take any $t\in [0,1]$ and any open neighbourhood $O(V,\delta)(\varphi(t))$ of $\varphi(t)$. We claim that $\varphi(s)\in O(V,\delta)(\varphi(t))$ whenever $\lvert s-t\rvert<\frac{\delta}{n},$ which implies that $\varphi$ is continuous. Without loss of generality, we assume that $s\leq t$. Fix a non-negative $k<n$. Then according to the definition of $\varphi(t)$ one can easily show that \[G_k=\{r\in [a_k,a_{k+1})\mid \varphi(s)(r)\neq \varphi(t)(r)\}\subseteq [b_{k,s},b_{k,t}),\]
where $b_{k,s}=a_k+s(a_{k+1}-a_k)$ and $b_{k,t}=a_k+t(a_{k+1}-a_k).$
Hence, \[O^1_k=\{r\in [a_k,a_{k+1})\mid \varphi(s)(r)\ominus \varphi(t)(r)\notin V\}\subseteq G_k \subseteq [b_{k,s},b_{k,t})\]
and
 \[O^2_k=\{r\in [a_k,a_{k+1})\mid \varphi(t)(r)\ominus \varphi(s)(r)\notin V\}\subseteq G_k \subseteq [b_{k,s},b_{k,t})\]
 This implies that
 \[\mu(O^1_k)\leq b_{k,t}-b_{k,s}=(t-s)((a_{k+1}-a_k))\leq(t-s)< \frac{\delta}{n} \]
 and
 \[\mu(O^2_k)\leq b_{k,t}-b_{k,s}=(t-s)((a_{k+1}-a_k))\leq(t-s)< \frac{\delta}{n}. \]
Thus, we obtain that
\[\mu\big(\{r\in J\mid \varphi(s)(r)\ominus \varphi(t)(r)\notin V\}\big)\leq\sum_{k=0}^{n-1}\mu(O^1_k)<n\times \frac{\delta}{n}=\delta\]
and
\[\mu\big(\{r\in J\mid \varphi(t)(r)\ominus\varphi(s)(r) \notin V\}\big)\leq\sum_{k=0}^{n-1}\mu(O^2_k)<n\times \frac{\delta}{n}=\delta\].
This implies that
$\varphi(s)\ominus \varphi(t)\in O(V,\delta)$ and $\varphi(t)\ominus \varphi(s) \in O(V,\delta),$
which implies that $\varphi(s)\in O(V,\delta)(\varphi(t))$.

The same argument applied to the whole MV-algebra $A^\bullet$ in place of $O(U,\varepsilon)$ implies the pathwise connectedness of $A^\bullet$.
\end{proof}

\begin{theorem}\label{the1}
For every Hausdorff topological MV-algebra $A$, there exists a natural topological
isomorphism $i_A:A\rightarrow A^\bullet$ of $A$ onto a closed subalgebra of the pathwise connected, locally
pathwise connected topological MV-algebra $A^\bullet$.
\end{theorem}

\begin{proof}
Define $i_A:A\rightarrow A^\bullet$ by $i_A(a)=a^\bullet$ for each $a\in A$, where $a^\bullet$ is defined by $a^\bullet(r)=a$ for each $r\in [0,1)$. It is obvious that $i_A$ is injective. Take any $a,b\in A$. Then $i_A(a\oplus b)=(a\oplus b)^\bullet$. That is,
\[(a\oplus b)^\bullet(r)=a\oplus b=a^\bullet(r)\oplus b^\bullet(r)=i_A(a)(r)\oplus i_A(b)(r)=(i_A(a)\oplus i_A(b))(r)\]
for each $r\in [0,1)$.
This means that $i_A(a\oplus b)=(i_A(a)\oplus i_A(b))$. Similarly, one can easily show that $i_A(0)=0^\bullet$ and $i_A(a^*)=i_A(a)^*$. Thus, we have proved that $i_A$ is an isomorphism from $A$ onto $i_A(A)$.

Let $\mathcal{V}$ be an open neighbourhood base at $0$ in $A$. To show that $i_A:A\rightarrow i_A(A)$ is homeomorphism, it is enough to show that $i_A(V)=O(V,\varepsilon)\cap i_{A}(A)$ for each $V\in \mathcal{V}$ and $1>\varepsilon>0$ by the fact that a homomorphism $f:B_1\rightarrow B_2$ of MV-algebras is continuous iff $f$ is continuous at $0$ of $B_1$ \cite[Proposition 3.7]{GLD}. Take each $a\in V$. Then $i_A(a)=a^\bullet$. Since $a^\bullet(r)=a\in V$ for each $r\in [0,1)$, we obtain that $\{r\in [0,1)\mid a^\bullet(r)\notin V\}=\emptyset$. Hence, $\mu\big(\{r\in [0,1)\mid a^\bullet(r)\notin V\}\big)=0<\varepsilon$. This implies that $a^\bullet\in O(V,\varepsilon)$. Thus we have shown that $i_A(V)\subseteq O(V,\varepsilon)\cap i_{A}(A)$. Take any $f\in O(V,\varepsilon)\cap i_{A}(A)$. We can assume that $f=a^\bullet$ for some $a\in A$. We claim that $a\in V$, which implies that $i_A(V)\supseteq O(V,\varepsilon)\cap i_{A}(A)$. Indeed, if not, then $a\notin V$. Thus, $\{r\in [0,1)\mid a^\bullet(r)\notin V\}=[0,1)$, which implies that \[\mu\big(\{r\in [0,1)\mid a^\bullet(r)\notin V\}\big)=1>\varepsilon.\]
This implies that $a^\bullet\notin O(V,\varepsilon)$, a contradiction with $a^\bullet\in O(V,\varepsilon)\cap i_A(A).$

Finally, we shall prove that $i_A(A)$ is closed in $A^\bullet$. Take any $f\in A^\bullet\setminus i_A(A)$. Then $f$ cannot be constant as a function from $[0,1)$ to $A$. Therefore, we can find real numbers $a_1,a_2,a_3,a_4$ satisfying $0<a_1<a_2\leq a_3<a_4<1$ and two distinct elements $x_1,x_2\in A$ such that $f$ is equal to $x_1$ on $[a_1,a_2)$ and $f$ is equal to $x_2$ on $[a_3,a_4)$. Since $A$ is Hausdorff and the sets $V(x)$, with $V\in\mathcal{V}$, form a neighbourhood base at $x$, there is $V\in \mathcal{V}$ such that $V(x_1)\cap V(x_2)=\emptyset$. Put $\varepsilon=\min\{a_2-a_1,a_4-a_3\}$. Shrinking the two intervals if necessary, we may assume that $\varepsilon=a_2-a_1=a_4-a_3$. Then we claim that $O(V,\frac{\varepsilon}{3})(f)\cap i_A(A)=\emptyset$, which implies that $i_A(A)$ is closed in $A^\bullet$.
Indeed, take any $g\in O(V,\frac{\varepsilon}{3})(f)$. Put \[E_1=\{r\in [0,1)\mid (g\ominus f)(r)\in V\}\]
and
\[E_2=\{r\in [0,1)\mid (f\ominus g)(r)\in V\}.\]
Since $g\in O(V,\frac{\varepsilon}{3})(f) $, we obtain that
\[\mu\big(\{r\in [0,1)\mid (g\ominus f)(r)\notin V\}\big)<\frac{\varepsilon}{3}\]
and
\[\mu\big(\{r\in [0,1)\mid (f\ominus g)(r)\notin V\}\big)<\frac{\varepsilon}{3}\]
which implies that $\mu(E_1)\geq 1-\frac{\varepsilon}{3}$ and $\mu(E_2)\geq 1-\frac{\varepsilon}{3}.$ We have the following Claim 1:

{\bf Claim 1:} $E_1\cap [a_1,a_2)\cap E_2\neq\emptyset$.
If not, then $E_1\cap [a_1,a_2)\cap E_2=\emptyset.$ That is, $(E_1\cap [a_1,a_2))\cap([a_1,a_2)\cap E_2)=\emptyset.$ Hence,
\[\mu\big(E_1\cap [a_1,a_2)\big)+\mu\big(E_2\cap [a_1,a_2)\big)=\mu\big((E_1\cap [a_1,a_2))\cup (E_2\cap [a_1,a_2))\big)\leq a_2-a_1=\varepsilon.\]
On the other hand, since $\mu\big(\{r\in [0,1)\mid (g\ominus f)(r)\notin V\}\big)<\frac{\varepsilon}{3}$, we obtain that
\[\mu(E_1\cap [a_1,a_2))=\mu(\{r\in [0,1)\mid (g\ominus f)(r)\in V\}\cap [a_1,a_2))>\frac{2\varepsilon}{3}.\] Similarly, we can obtain that
\[\mu(E_2\cap [a_1,a_2))=\mu(\{r\in [0,1)\mid (f\ominus g)(r)\in V\}\cap [a_1,a_2))>\frac{2\varepsilon}{3}\]
 by $\mu\big(\{r\in [0,1)\mid (f\ominus g)(r)\notin V\}\big)<\frac{\varepsilon}{3}.$ Hence,
\[\mu\big(E_1\cap [a_1,a_2)\big)+\mu\big(E_2\cap [a_1,a_2)\big)>\frac{2\varepsilon}{3}+\frac{2\varepsilon}{3}=\frac{4\varepsilon}{3}>\varepsilon.\]
This is a contradiction with $\mu\big(E_1\cap [a_1,a_2)\big)+\mu\big(E_2\cap [a_1,a_2)\big)\leq \varepsilon.$

Also, we have the following Claim 2:

{\bf Claim 2:} $E_1\cap [a_3,a_4)\cap E_2\neq\emptyset$.

Indeed, put $\delta=a_4-a_3$. If $E_1\cap [a_3,a_4)\cap E_2=\emptyset$, then $(E_1\cap [a_3,a_4))\cap ([a_3,a_4)\cap E_2)=\emptyset$. Hence,
\[\mu\big(E_1\cap [a_3,a_4)\big)+\mu\big(E_2\cap [a_3,a_4)\big)=\mu\big((E_1\cap [a_3,a_4))\cup (E_2\cap [a_3,a_4))\big)\leq a_4-a_3=\delta.\]
On the other hand, since $\mu\big(\{r\in [0,1)\mid (g\ominus f)(r)\notin V\}\big)<\frac{\varepsilon}{3}\leq \frac{\delta}{3}$, we obtain that
\[\mu(E_1\cap [a_3,a_4))=\mu(\{r\in [0,1)\mid (g\ominus f)(r)\in V\}\cap [a_3,a_4))>\frac{2\delta}{3}.\] Similarly, we can obtain that
\[\mu(E_2\cap [a_3,a_4))=\mu(\{r\in [0,1)\mid (f\ominus g)(r)\in V\}\cap [a_3,a_4))>\frac{2\delta}{3}\] by $\mu\big(\{r\in [0,1)\mid (f\ominus g)(r)\notin V\}\big)<\frac{\varepsilon}{3}\leq\frac{\delta}{3}.$ Hence,
\[\mu\big(E_1\cap [a_3,a_4)\big)+\mu\big(E_2\cap [a_3,a_4)\big)>\frac{2\delta}{3}+\frac{2\delta}{3}=\frac{4\delta}{3}>\delta.\]
This is a contradiction with $\mu\big(E_1\cap [a_3,a_4)\big)+\mu\big(E_2\cap [a_3,a_4)\big)\leq \delta.$

According to Claim 1, we can take a point $r_1\in E_1\cap [a_1,a_2)\cap E_2$. Noting that $f$ is equal to $x_1$ on $[a_1,a_2)$, we obtain that
 \[g(r_1)\ominus x_1=g(r_1)\ominus f(r_1)=(g\ominus f)(r_1)\in V\]
 and
 \[ x_1 \ominus g(r_1)=f(r_1)\ominus g(r_1)=(f\ominus g)(r_1)\in V\]
 by the definitions $E_1$ and $E_2$ above.
 This implies that $g(r_1)\in V(x_1)$.

 Similarly, we can take a point $r_2\in E_1\cap [a_3,a_4)\cap E_2$ by Claim 2. Then as the same way one can obtain that $g(r_2)\in V(x_2)$,
 because $f$ is equal to $x_2$ on $[a_3,a_4)$. Therefore, $g(r_1)\neq g(r_2)$ by $V(x_1)\cap V(x_2)=\emptyset$. Finally we have proved that $g$ cannot be constant as a function on $[0,1)$,
 so we have proved that $O(V,\frac{\varepsilon}{3})(f)\cap i_A(A)=\emptyset$. This implies that $i_A(A)$ is closed in $A^\bullet$.

 $A^\bullet$ is a pathwise connected, locally pathwise connected topological MV-algebra by Proposition \ref{Pro4}. Therefore, we complete the proof.
\end{proof}

In what follows we identify a topological MV-algebra $A$ with its image $i_A(A) \subseteq A^\bullet$ defined
in Theorem \ref{the1}. Let us show that the MV-algebra $A$ is placed in $A^\bullet$ in a very special way,
permitting an extension of continuous bounded pseudometrics from $A$ over $A^\bullet$.

\begin{definition}\label{Def3.5}
Let $X$ be a set. A function $d:X\times X\rightarrow [0,+\infty)$ is called a pseudometric on $X$ if $d$ satisfies the following conditions:
\begin{enumerate}
\item[(1)] $d(x,y)=0$ for $x=y$;
\item[(2)] $d(x,y)=d(y,x)$ for each $x,y\in X$;
\item[(3)] $d(x,y)\leq d(x,z)+d(z,y)$ for each $x,y,z\in X$.
\end{enumerate}
\end{definition}

\begin{theorem}\label{the2}
Let $d$ be a continuous bounded pseudometric on a topological $MV$-algebra $A$. Then $d$ admits an extension to a
continuous bounded pseudometric $d^\bullet$ over the $MV$-algebra $A^\bullet$. In addition, if $d$ is
a metric on $A$ generating the topology of $A$, then $d^\bullet$ is also a metric on $A^\bullet$ generating the
topology of $A^\bullet$.
\end{theorem}

\begin{proof}

Without loss of generality we can assume that $d$ is bounded by $1$. Take any $f,g\in A^\bullet$. Then there is a partition \(0=a_0<a_1<\dots<a_n=1\) of $J=[0,1)$
such that for each non-negative $k<n$ both $f$ and
$g$ are constant on each interval \(J_k=[a_{k},a_{k+1})\) and are equal to \(x_k,y_k\in A\) on this interval, respectively. We define $d^\bullet(f,g)$ as follows:
\[d^\bullet(f,g)=\sum_{k=0}^{n-1}(a_{k+1}-a_k)d(x_k,y_k).\]

Since $f$ and $g$ are constant on each interval \(J_k=[a_{k},a_{k+1})\), one can easily verify that $d^\bullet(f,g)$ does not depend on the choice of
the partition \(0=a_0<a_1<\dots<a_n=1\) of $J=[0,1)$. Now we prove that $d^\bullet$ satisfies Definition \ref{Def3.5}. If $f=g$, then $x_k=y_k$ for each
non-negative $k< n$. Hence, $d(x_k,y_k)=0$ by $d$ being a pseudometric. This implies that $d^\bullet(f,g)=0$. Since $d$ is symmetric, it
is obvious that $d^\bullet(f,g)=d^\bullet(g,f)$. Take any $h\in A^\bullet$. We can assume that $h$ is constant on each interval \(J_k=[a_{k},a_{k+1})\)
 and is equal to \(z_k\in A\) on this interval for each non-negative $k< n$. For each non-negative $k< n$, $d(x_k,y_k)\leq d(x_k,z_k)+d(z_k,y_k)$ by $d$
 satisfying (3) in Definition \ref{Def3.5}, so

 \begin{align*}
d^\bullet(f,g)&=\sum_{k=0}^{n-1}(a_{k+1}-a_k)d(x_k,y_k)
\\&\leq \sum_{k=0}^{n-1}(a_{k+1}-a_k)(d(x_k,z_k)+d(z_k,y_k))
\\&=\sum_{k=0}^{n-1}(a_{k+1}-a_k)d(x_k,z_k)+\sum_{k=0}^{n-1}(a_{k+1}-a_k)d(z_k,y_k)
\\&=d^\bullet(f,h)+d^\bullet(h,g).
\end{align*}
This proves that $d^\bullet$ is a pseudometric on $A^\bullet$.  $d^\bullet$ is bounded by $d$ being bounded.

Let us show that $d^\bullet$ is continuous on $A^\bullet\times A^\bullet$. Let $\varepsilon>0$ and take any $(f,g)\in A^\bullet\times A^\bullet$. We can
assume that for each non-negative $k< n$ both $f$ and
$g$ are constant on each interval \(J_k=[a_{k},a_{k+1})\) and are equal to \(x_k,y_k\in A\) on this interval, respectively.
Since $d$ is continuous, there is an open neighbourhood $V$ of $0$ in $A$ such that $d(x_k,x)< \frac{\varepsilon}{4}$ for each $x\in V(x_k)$
and $d(y_k,y)< \frac{\varepsilon}{4}$ for each $y\in V(y_k)$ for each non-negative $k<n$. Now we claim that
 \[\mid d^\bullet(f_1,g_1)-d^\bullet(f,g)\mid<\varepsilon\]
whenever $f_1\in O(V,\frac{\varepsilon}{4})(f)$ and $g_1\in O(V,\frac{\varepsilon}{4})(g)$. This implies that $d^\bullet$ is continuous.

Indeed, we can assume without loss of generality that for each non-negative $k< n$ both $f_1$ and
$g_1$ are constant on each interval \(J_k=[a_{k},a_{k+1})\) and are equal to \(x_k^1,y_k^1\in A\) on this interval, respectively.
Denote by $L_1$ and $L_2$ the sets of all integers $k\leq n-1$ such that $x_k^1\in V(x_k)$ and $y_k^1\in V(y_k)$, respectively. Put
$M_1=\{0,1,\ldots,n-1\}\setminus L_1$ and $M_2=\{0,1,\ldots,n-1\}\setminus L_2$. It follows from the choice of $d$ and $f_1$ that $d(x_k^1,x_k)<\frac{\varepsilon}{4}$
for each $k\in L_1$ and $\sum_{k\in M_1}(a_{k+1}-a_k)<\frac{\varepsilon}{4}$. Similarly, we have that $d(y_k^1,y_k)<\frac{\varepsilon}{4}$
for each $k\in L_2$ and $\sum_{k\in M_2}(a_{k+1}-a_k)<\frac{\varepsilon}{4}$. Since $d$ is bounded by $1$, $\sum_{k\in L_1}(a_{k+1}-a_k)\leq 1$ and
$\sum_{k\in L_2}(a_{k+1}-a_k)\leq 1$. Then

\begin{align*}
\mid d^\bullet(f_1,g_1)-d^\bullet(f,g)\mid&=\mid d^\bullet(f_1,g_1)-d^\bullet(f,g_1)+d^\bullet(f,g_1)-d^\bullet(f,g)\mid
\\&\leq \mid d^\bullet(f_1,g_1)-d^\bullet(f,g_1)\mid+\mid d^\bullet(f,g_1)-d^\bullet(f,g)\mid
\\&\leq d^\bullet(f_1,f)+d^\bullet(g_1,g)
\\&=\sum_{k=0}^{n-1}(a_{k+1}-a_k)d(x_k^1,x_k)+\sum_{k=0}^{n-1}(a_{k+1}-a_k)d(y_k^1,y_k)
\\&=\sum_{k\in L_1}(a_{k+1}-a_k)d(x_k^1,x_k)+\sum_{k\in M_1}(a_{k+1}-a_k)d(x_k^1,x_k)
\\&+\sum_{k\in L_2}(a_{k+1}-a_k)d(y_k^1,y_k)+\sum_{k\in M_2}(a_{k+1}-a_k)d(y_k^1,y_k)
\\&\leq \sum_{k\in L_1}(a_{k+1}-a_k)d(x_k^1,x_k)+\sum_{k\in M_1}(a_{k+1}-a_k)
\\&+\sum_{k\in L_2}(a_{k+1}-a_k)d(y_k^1,y_k)+\sum_{k\in M_2}(a_{k+1}-a_k)
\\&<\max\{d(x_k^1,x_k)\mid k\in L_1\}+\frac{\varepsilon}{4}
\\&+\max\{d(y_k^1,y_k)\mid k\in L_2\}+\frac{\varepsilon}{4}
\\&<\frac{\varepsilon}{4}+\frac{\varepsilon}{4}+\frac{\varepsilon}{4}+\frac{\varepsilon}{4}=\varepsilon.
\end{align*}
This proves the continuity of $d^\bullet$ on $A^\bullet$.

Finally, suppose that $d$ is a metric on $A$ generating the topology of $A$. According to the definition of $d^\bullet$, one can
easily obtain that $d^\bullet(f,g)>0$ if $f,g\in A^\bullet$ are distinct, so that $d^\bullet$ is a metric on $A^\bullet$. Now we shall prove
 that $d^\bullet$ generates the topology of $A^\bullet$.

 Take any $f\in  A^\bullet$ and an open neighbourhood $O(V,\varepsilon)(f)$ of $f$ in $A^\bullet$. Let $0=a_0<a_1<\ldots<a_n=1$ be a
 partition of $J=[0,1)$ such that $f$ takes a constant value $x_k$ on each $[a_k,a_{k+1})$. Since $V(x_k)$ is an open neighbourhood of $x_k$ in $A$ for each
 non-negative integer $k<n$, one can find $\delta>0$ such that $\{y\in A\mid d(x_k,y)<\delta\}\subseteq V(x_k)$ for each $k=0,1,\ldots,n-1$. Put $\delta_0=\varepsilon\delta.$
 We claim that
 \[\{g\in A^\bullet\mid d^\bullet(f,g)< \delta_0\}\subseteq O(V,\varepsilon)(f),\]
 which implies that $d^\bullet$ generates the topology of $A^\bullet$.

 Indeed, suppose that an element $g\in A^\bullet$ satisfies $d^\bullet(f,g)< \delta_0$. We can assume without loss of
 generality that $g$ takes a constant value $y_k$ on $[a_k,a_{k+1})$ for each non-negative $k< n$. Denote by $P$ the set of all non-negative
 integers $k<n$ such that $y_k\notin V(x_k)$. Clearly, $d(y_k,x_k)\geq\delta$ holds for each $k\in P$. Thus,

 \begin{align*}
 \sum_{k\in P}(a_{k+1}-a_k)\delta &\leq \sum_{k\in P}(a_{k+1}-a_k)d(y_k,x_k)
 \\&\leq \sum_{i<n}(a_{i+1}-a_i)d(y_i,x_i)
 \\&=d^\bullet(f,g)< \delta_0.
 \end{align*}
It follows that
\[\sum_{k\in P}(a_{k+1}-a_k)<\frac{\delta_0}{\delta}=\varepsilon.\]
 For each $r\in [0,1)\setminus \bigcup_{k\in P}[a_k,a_{k+1})$, there is $k\in \{0,1,\ldots,n-1\}\setminus P$
such that $r\in [a_k,a_{k+1})$. According to the definition of $P$, we obtain that $g(r)=y_k\in V(x_k)=V(f(r))$, that is, $g(r)\ominus f(r)\in V$
and $f(r)\ominus g(r)\in V$. Hence, we obtain that
\[\{r\in J\mid (g\ominus f)(r)=g(r)\ominus f(r)\notin V\}\subseteq \bigcup_{k\in P}[a_k,a_{k+1})\]
and
\[\{r\in J\mid (f\ominus g)(r)=f(r)\ominus g(r)\notin V\}\subseteq \bigcup_{k\in P}[a_k,a_{k+1})\].
This implies that
\[\mu\big(\{r\in J\mid (g\ominus f)(r))\notin V\}\big)\leq \sum_{k\in P}(a_{k+1}-a_k)<\varepsilon\]
and
\[\mu\big(\{r\in J\mid (f\ominus g)(r))\notin V\}\big)\leq \sum_{k\in P}(a_{k+1}-a_k)<\varepsilon\].
Hence, $g\ominus f\in O(V,\varepsilon)$ and $f\ominus g\in O(V,\varepsilon)$, which is equivalent to $g\in O(V,\varepsilon)(f).$
Thus, we have proved that
\[\{g\in A^\bullet\mid d^\bullet(f,g)< \delta_0\}\subseteq O(V,\varepsilon)(f).\]
\end{proof}

The following result gives an extension of bounded continuous real-valued functions.

\begin{corollary}\label{cor1}
Let $h$ be a continuous real-valued bounded function on a topological MV-algebra $A$. Then $h$ admits an extension to a bounded continuous function
on the MV-algebra $A^\bullet$.
\end{corollary}

\begin{proof}
We may assume that $\lvert h\rvert \leq M$ for some $M>0$. For $f\in A^\bullet$, choose a partition $0=a_0<a_1<\cdots<a_n=1$ such that $f$ is constant on each interval $[a_k,a_{k+1})$, say $f(r)=x_k$ on $[a_k,a_{k+1})$. Define
\[
h^\bullet(f)=\sum_{k=0}^{n-1}(a_{k+1}-a_k)h(x_k).
\]
This definition is independent of the chosen partition, since passing to a common refinement does not change the above sum. If $a\in A$, then $h^\bullet(a^\bullet)=h(a)$, so $h^\bullet$ extends $h$.

It remains to prove continuity. Fix $f\in A^\bullet$ and $\varepsilon>0$. Choose a partition $0=a_0<a_1<\cdots<a_n=1$ such that $f(r)=x_k$ on $[a_k,a_{k+1})$. Since $h$ is continuous, for each $k$ there is a neighbourhood $V_k$ of $0$ such that $\lvert h(y)-h(x_k)\rvert<\varepsilon/2$ whenever $y\in V_k(x_k)$. Choose a neighbourhood $V$ of $0$ with $V\subseteq\bigcap_{k=0}^{n-1}V_k$ and put $\eta=\varepsilon/(4M+1)$. If $g\in O(V,\eta)(f)$ and $g(r)=y_k$ on the same refined partition, then the set of $r$ for which $g(r)\notin V(f(r))$ has measure less than $\eta$. Hence
\[
\lvert h^\bullet(g)-h^\bullet(f)\rvert
\leq \frac{\varepsilon}{2}+2M\eta<\varepsilon.
\]
Thus $h^\bullet$ is continuous and bounded.
\end{proof}

\begin{lemma}\cite[Theorem 3.6]{GLD} \label{Lemma2}
Let $A$ be a topological MV-algebra. If $\mathcal{V}$ is an open neighbourhood base at $a$, then
 $\{U^{(a)}\mid U\in \mathcal{V}\}$ is an open neighbourhood base at $0$.
\end{lemma}

Let $f:X\rightarrow Y$ be a function of topological spaces. We say that $f$ is open at the point $x\in X$ if for each neighbourhood
$U_x$ of $x$, $f(U_x)$ is a neighbourhood of $f(x)$. It is obvious that $f$ is an open function if and only if $f$ is open at each point of $X$.
\begin{proposition}\label{Pro3.9}
Let $A$ and $B$ be topological MV-algebras and $f:A\rightarrow B$ a homomorphism of MV-algebras.  Denote the zero elements of the MV-algebras $A$ and $B$, respectively, by $0_A$ and $0_B$. Then the following statements are equivalent:
\begin{enumerate}
\item[(1)] $f$ is an open function;
\item[(2)] $f$ is open at $0_A$;
\item[(3)] $f$ is open at $a$, for some $a\in A$.
\end{enumerate}
\end{proposition}

\begin{proof}
 It is clear that $(1)\Rightarrow(2)$ and $(2)\Rightarrow(3)$.

 $(3)\Rightarrow(2)$. If $a=0_A$, there is nothing to prove. Assume that $a\neq 0_A$, and take any open neighbourhood $V$ of $0_A$.  Since $A$ is a topological MV-algebra, there is an
 open neighbourhood $W$ of $0_A$ such that $W\oplus W\subseteq V$. Then $W(a)$ is a neighbourhood of $a$ by Proposition \ref{Pro1}. Since $f$ is a homomorphism of MV-algebras, one can easily show that
  $f(W(a))\subseteq f(W)(f(a))$. Indeed, take any $y\in W(a)$. Then $y\ominus a\in W$ and $a\ominus y\in W$. $f$ is a homomorphism of MV-algebras, so,
  \[f(y)\ominus f(a)=f(y\ominus a)\in f(W)\]
   and
   \[f(a)\ominus f(y)=f(a\ominus y)\in f(W),\]
   which implies that $f(y)\in f(W)(f(a))$. Since
 $f$ is open at $a$, $f(W(a))$ is a neighbourhood of $f(a)$. Hence, $f(W)(f(a))$ is a neighbourhood of $f(a)$. Put $O=f(W)(f(a))$.
 Then we claim that $O^{(f(a))}\subseteq f(V)$. Take any $z\in O^{(f(a))}$. Then,
 \[z\vee f(a)\in O=f(W)(f(a))\]
 and
 \[f(a)\ominus z\in  O=f(W)(f(a)). \]
Hence, by Lemma \ref{Lema1},
\begin{align*}
z\ominus f(a)&=z\odot f(a)^\ast
\\&= (z\odot f(a)^\ast)\vee 0
\\&=(z\odot f(a)^\ast)\vee (f(a)\odot f(a)^\ast)
\\&= (z\vee f(a))\odot f(a)^\ast
\\&=(z\vee f(a))\ominus f(a) \in f(W)
\end{align*}
and
\begin{align*}
f(a)\wedge z&=f(a)\odot (f(a)^\ast \oplus z)
\\&=f(a)\ominus (f(a)^\ast \oplus z)^\ast
\\&=f(a)\ominus (f(a) \odot z^\ast)
\\&=f(a)\ominus( f(a)\ominus z)\in f(W)
\end{align*}
Since $W\oplus W\subseteq V$ and $f$ is a homomorphism of MV-algebras, we obtain that $f(W)\oplus f(W)\subseteq f(V)$. Hence,
\[z=(f(a)\wedge z)\oplus(z\ominus f(a))\in f(W)\oplus f(W)\subseteq f(V).\] This completes the proof of $O^{(f(a))}\subseteq f(V).$
By Lemma \ref{Lemma2}, $O^{(f(a))}$ is a neighbourhood of $f(0_A)$, because $O=f(W)(f(a))$ is a neighbourhood of $f(a)$. Thus, $f(V)$
 is a neighbourhood of $f(0_A)$. This shows that $f$ is open at $0_A$.

 $(2)\Rightarrow(3).$  It is enough to show that $f$ is open at any $a\in A$. Take an open neighbourhood $U$ of $a$. Since $0_A\oplus a\in U$ and $A$ is a
 topological MV-algebra, there are an open neighbourhood $V$ of $0_A$ and an open neighbourhood $W$ of $a$ such that $V\oplus W\subseteq U$.
 By $f$ being a homomorphism, one can obtain that $f(V)\oplus f(W)\subseteq f(U)$. Since $W$ is an open neighbourhood of $a$, $W^{(a)}$ is
 an open neighbourhood of $0_A$. From the fact that $f$ is open at $0_A$ it follows that both $f(V)$ and $f(W^{(a)})$ are neighbourhoods of $f(0_A)$.
 Then we can take an open neighbourhood $O$ of $f(0_A)$ such that $O\subseteq f(V)\cap f(W^{(a)})$. We claim that $O(f(a))\subseteq f(U)$, which implies that $f(U)$
 is a neighbourhood of $f(a)$, because so is $O(f(a))$. This shows that $f$ is open at $a$.

 Indeed, take any $z\in O(f(a))$. Then, by $f$ being a homomorphism,
 \[ f(a)\ominus z\in O\subseteq f(V)\cap f(W^{(a)})\subseteq f(V)\cap f(W)^{(f(a))} \]
 and
 \[z\ominus f(a)\in O\subseteq f(V)\cap f(W^{(a)})\subseteq f(V)\cap f(W)^{(f(a))}. \]
Because
\begin{align*}
f(a)\wedge z&=f(a)\odot (f(a)^\ast \oplus z)
\\&=f(a)\ominus (f(a)^\ast \oplus z)^\ast
\\&=f(a)\ominus (f(a) \odot z^\ast)
\\&=f(a)\ominus( f(a)\ominus z),
\end{align*}
we obtain that
\[f(a)\wedge z=f(a)\ominus( f(a)\ominus z)\in f(W)\]
 by $f(a)\ominus z\in f(V)\cap f(W)^{(f(a))}\subseteq f(W)^{(f(a))}.$
 Noting that $z\ominus f(a)\in f(V)$, we obtain that
 \[z=(f(a)\wedge z)\oplus (z\ominus f(a))\in f(W)\oplus f(V)\subseteq f(U) \]
 by (1) in Lemma \ref{Lema1}. This implies that $O(f(a))\subseteq f(U)$, which completes the proof of the claim above.
\end{proof}

\begin{theorem}\label{The3}
Let $\varphi:A_1\rightarrow A_2$ be a continuous homomorphism of topological MV-algebras. Then $\varphi$ admits a natural extension to a
continuous homomorphism $\varphi^\bullet:A_1^\bullet\rightarrow A_2^\bullet$. In addition, if $\varphi$ is open and onto, then so is $\varphi^\bullet$.
\end{theorem}

\begin{proof}
For each $f\in A_1^\bullet$, define an element $\varphi^\bullet(f)\in A_2^\bullet$ by $\varphi^\bullet(f)(r)=\varphi(f(r))$ for each $r\in J=[0,1).$
Take $f_1,f_2\in A_1^\bullet$ and $r\in J$. Then
 \begin{align*}
\varphi^\bullet(f_1\oplus f_2)(r)&=\varphi((f_1\oplus f_2)(r))
\\&=\varphi(f_1(r)\oplus f_2(r))
\\&=\varphi(f_1(r))\oplus \varphi(f_2(r))
\\&=\varphi^\bullet(f_1)(r)\oplus \varphi^\bullet(f_2)(r)
\\&=(\varphi^\bullet(f_1)\oplus \varphi^\bullet(f_2))(r)
 \end{align*}
This implies that $\varphi^\bullet(f_1\oplus f_2)=\varphi^\bullet(f_1)\oplus \varphi^\bullet(f_2).$

Similarly,
\begin{align*}
\varphi^\bullet(f^\ast)(r)&=\varphi(f^\ast(r))
\\&=\varphi((f(r))^\ast)
\\&=\varphi((f(r)))^\ast
\\&=(\varphi^\bullet(f)(r))^\ast
\\&=\varphi^\bullet(f)^\ast(r),
\end{align*}
which implies that $\varphi^\bullet(f^\ast)=\varphi^\bullet(f)^\ast$. Thus, we have proved that $\varphi^\bullet$ is homomorphism.

Next, we shall show that $\varphi^\bullet$ is an extension of $\varphi$. Let $i_{A_i}:A_i\rightarrow A_i^\bullet$ be the canonical embedding defined by $i_{A_i}(a)=a^\bullet$ for $i=1,2$. Then for each $a\in A_1$ and each $r\in [0,1)$, we obtain that
\begin{align*}
\varphi^\bullet(i_{A_1}(a))(r)&=\varphi^\bullet(a^\bullet)(r)
\\&=\varphi((a^\bullet)(r))
\\&=\varphi(a)
\\&=\varphi(a)^\bullet(r)
\\&=i_{A_2}(\varphi(a))(r),
\end{align*}
which implies that $\varphi^\bullet(i_{A_1}(a))=i_{A_2}(\varphi(a))$. Thus we have proved that $\varphi^\bullet$ is an extension of $\varphi$.

If $\varphi$ is continuous, then to show that $\varphi^\bullet$ is continuous it is enough to show that $\varphi^\bullet$ is continuous at $0_{A_1}^\bullet$
by \cite[Proposition 3.7]{GLD}. Take any neighbourhood $O(V,\varepsilon)$ of $\varphi^\bullet(0_{A_1}^\bullet)$. Since $\varphi$ is continuous at $0_{A_1}$ and $V$ is a
 neighbourhood of $\varphi(0_{A_1})$, there is a neighbourhood $U$ of $0_{A_1}$ such that $\varphi(U)\subseteq V$. We claim that the neighbourhood $O(U,\varepsilon)$
 of $0_{A_1}^\bullet$ satisfies $\varphi^\bullet(O(U,\varepsilon))\subseteq O(V,\varepsilon).$

 Take any $f\in O(U,\varepsilon)$. Put
 \[E_1=\{r\in [0,1)\mid f(r)\notin U\}\]
 and
 \[E_2=\{r\in [0,1)\mid \varphi(f(r))\notin V\}\]

 If $r\in E_2$ and $r\notin E_1$, then $f(r)\in U$. Thus, by
$\varphi(U)\subseteq V $, we obtain that $\varphi(f(r))\in V$, which implies that $r\notin E_2$. This is a contradiction with $r\in E_2$. Hence,
 $E_2\subseteq E_1$. Noting that $f\in O(U,\varepsilon)$, $\mu(E_2)\leq \mu(E_1)<\varepsilon$. This implies that
 \[\varphi^\bullet (f)=\varphi\circ f\in O(V,\varepsilon).\]
 Thus, we have proved that $\varphi^\bullet$ is continuous.

Finally, we shall show that $\varphi^\bullet$ is open when $\varphi$ is open and onto. It is obvious that $\varphi^\bullet$ is onto. In fact, fixing an $f\in A_2^\bullet$, since $\varphi$ is onto and $f$ is a step function, $f$ has only finitely many values. For each value $b$ of $f$, choose $a_b\in\varphi^{-1}(b)$ and define $g:J\rightarrow A_1$ by $g(r)=a_{f(r)}$. Then $g\in A_1^\bullet$, and one can easily show that $\varphi^\bullet(g)=f$. This implies that $\varphi^\bullet$ is onto.
According to Proposition \ref{Pro3.9} it is enough to show that
$\varphi^\bullet$ is open at $0_{A_1}^\bullet$. Take any open neighbourhood $O(V,\varepsilon)$ of $0_{A_1}^\bullet$ in $A_1^\bullet$. Since $\varphi$ is open
and $V$ is an open neighbourhood of $0_{A_1}$ in $A_1$, $\varphi(V)$ is an open neighbourhood of $\varphi(0_{A_1})$ in $A_2$. It is obvious that
$O(\varphi(V),\varepsilon)$ is an open neighbourhood of $\varphi(0_{A_1})^\bullet$ in $A_2^\bullet$. We claim that
$O(\varphi(V),\varepsilon)\subseteq \varphi^\bullet(O(V,\varepsilon))$, which implies that $\varphi^\bullet$ is open at $0_{A_1}^\bullet$.

Indeed, take any $f\in O(\varphi(V),\varepsilon)$. Then
\[\mu\big(\{r\in[0,1)\mid f(r)\notin \varphi(V)\}\big)<\varepsilon\]
and there is a $g\in A_1^\bullet$ such that \[f=\varphi^\bullet(g)=\varphi\circ g.\]
Since $g$ is a step function, it has only finitely many values. For each value $c$ of $g$, if $\varphi^{-1}(\varphi(c))\cap V\neq \emptyset$, then fix an element $a_c\in \varphi^{-1}(\varphi(c))\cap V$; otherwise, fix an element $a_c\in \varphi^{-1}(\varphi(c))$.
Define $h_g:[0,1)\rightarrow A_1$ by $h_g(r)=a_{g(r)}$. Then $h_g\in A_1^\bullet$. It is obvious that
 \[\varphi^\bullet(h_g)(r)=\varphi(h_g(r))=\varphi(a_{g(r)})=\varphi(g(r))=f(r)\]
 holds for each $r\in [0,1).$ This implies that $\varphi^\bullet(h_g)=f.$
 According to the definition of $h_g$, if $f(r)\in \varphi(V)$, then $h_g(r)\in V$; if $f(r)\notin \varphi(V)$, then no element of $\varphi^{-1}(f(r))$ belongs to $V$. Hence
 \[\{r\in[0,1)\mid h_g(r)\notin V\}\subseteq \{r\in[0,1)\mid f(r)\notin \varphi(V)\}.\]
 Then
 \[\mu\big(\{r\in[0,1)\mid h_g(r)\notin V\}\big)\leq \mu\big(\{r\in[0,1)\mid f(r)\notin \varphi(V)\}\big)<\varepsilon,\]
 which implies that $h_g\in O(V,\varepsilon).$ This completes the proof of $O(\varphi(V),\varepsilon)\subseteq \varphi^\bullet(O(V,\varepsilon))$.
\end{proof}

In what follows, a continuous quotient homomorphism means a continuous
homomorphism of topological \(MV\)-algebras which is also a quotient map in
the topological sense. 

\begin{lemma}\cite[Theorem 3.3]{XieYang}\label{Lem:quotient-open}
Let \(\varphi:A_1\rightarrow A_2\) be a continuous quotient homomorphism of
topological \(MV\)-algebras. Then \(\varphi\) is an open mapping.
\end{lemma}

\begin{corollary}\label{Cor:HM-preserves-quotient}
The Hartman--Mycielski construction preserves continuous quotient
homomorphisms. More precisely, let
\[
\varphi:A_1\rightarrow A_2
\]
be a continuous quotient homomorphism of topological \(MV\)-algebras. Then
\[
\varphi^\bullet:A_1^\bullet\rightarrow A_2^\bullet
\]
is also a continuous quotient homomorphism.
\end{corollary}

\begin{proof}
Since \(\varphi\) is a quotient homomorphism, it is onto and continuous. By
Lemma \ref{Lem:quotient-open}, \(\varphi\) is open. Therefore, by Theorem
\ref{The3}, the induced homomorphism
\[
\varphi^\bullet:A_1^\bullet\rightarrow A_2^\bullet
\]
is continuous, open and onto. Hence \(\varphi^\bullet\) is a continuous open
surjection. Every continuous open surjection is a quotient map. Consequently,
\(\varphi^\bullet\) is a continuous quotient homomorphism.
\end{proof}

\begin{corollary}\label{Cor:HM-commutes-quotients}
Let \(\varphi:A_1\rightarrow A_2\) be a continuous quotient homomorphism of
topological \(MV\)-algebras, and let
\[
I=\ker\varphi .
\]
Put
\[
I^\bullet=\{f\in A_1^\bullet\mid f(r)\in I\ \text{for all }r\in J\}.
\]
Then
\[
\ker\varphi^\bullet=I^\bullet.
\]
Moreover,
\[
A_1^\bullet/I^\bullet\cong A_2^\bullet
\]
as topological \(MV\)-algebras. In particular, if
\(\pi_I:A\rightarrow A/I\) is the canonical quotient homomorphism and \(A/I\)
is endowed with the quotient topology, then there is a natural topological
\(MV\)-isomorphism
\[
(A/I)^\bullet\cong A^\bullet/I^\bullet .
\]
Thus the Hartman--Mycielski construction commutes with taking quotients.
\end{corollary}

\begin{proof}
For each \(f\in A_1^\bullet\), we have
\[
f\in\ker\varphi^\bullet
\]
if and only if
\[
\varphi^\bullet(f)=0_{A_2}^\bullet.
\]
By the definition of \(\varphi^\bullet\), this is equivalent to
\[
\varphi(f(r))=0_{A_2}
\quad\text{for all }r\in J.
\]
Equivalently,
\[
f(r)\in \ker\varphi=I
\quad\text{for all }r\in J.
\]
Hence
\[
\ker\varphi^\bullet=I^\bullet.
\]
It is also immediate from the pointwise operations on \(A_1^\bullet\) that
\(I^\bullet\) is an ideal of \(A_1^\bullet\).

By Corollary \ref{Cor:HM-preserves-quotient},
\[
\varphi^\bullet:A_1^\bullet\rightarrow A_2^\bullet
\]
is a continuous quotient homomorphism. Therefore, by the homomorphism theorem
for \(MV\)-algebras \cite[Corollary 3.7]{XieYang}, \(\varphi^\bullet\) induces a bijective \(MV\)-homomorphism
\[
\overline{\varphi^\bullet}:A_1^\bullet/I^\bullet\rightarrow A_2^\bullet
\]
defined by
\[
\overline{\varphi^\bullet}([f]_{I^\bullet})=\varphi^\bullet(f).
\]
Let
\[
q:A_1^\bullet\rightarrow A_1^\bullet/I^\bullet
\]
be the canonical quotient map. Then
\[
\varphi^\bullet=\overline{\varphi^\bullet}\circ q.
\]
Since \(q\) is a quotient map and \(\varphi^\bullet\) is continuous,
\(\overline{\varphi^\bullet}\) is continuous. Conversely, since
\(\varphi^\bullet\) is a quotient map and
\[
q=(\overline{\varphi^\bullet})^{-1}\circ \varphi^\bullet,
\]
the continuity of \(q\) implies that
\((\overline{\varphi^\bullet})^{-1}\) is continuous. Therefore
\(\overline{\varphi^\bullet}\) is a homeomorphism. Since it is also an
\(MV\)-isomorphism, it is a topological \(MV\)-isomorphism.

Taking \(\varphi=\pi_I:A\rightarrow A/I\), we obtain
\[
(A/I)^\bullet\cong A^\bullet/I^\bullet.
\]
This proves that the Hartman--Mycielski construction commutes with taking
quotients.
\end{proof}



{\bf Conclusion:} It is well known that MV-algebras play an important role in investigating the algebraic structures of logical systems.
Recently, several mathematicians have studied a number of algebraic structures associated with logical systems endowed with a topology.
Undoubtedly, topological MV-algebras, introduced by Hoo \cite{Hoo}, are among the most important structures in this context. In this study,
we mainly show how, for every Hausdorff topological MV-algebra $A$, one can topologically and isomorphically embed $A$ into a pathwise connected,
locally pathwise connected topological MV-algebra $A^\bullet$. This provides a method for constructing a new topological MV-algebra from an existing one.
In future work, we plan to investigate the connections between the algebraic and topological properties of the topological MV-algebra $A$ and
 its associated topological MV-algebra $A^\bullet$.


\def\cprime{$'$}

\end{document}